\documentclass[12pt,a4paper]{article}
\usepackage{amssymb}
\usepackage{amsmath}
\usepackage[symbol]{footmisc}
\usepackage{amsthm}
\usepackage{amscd}
\usepackage{multirow}
\usepackage{amstext}
\usepackage{booktabs}
\usepackage{vmargin}
\usepackage{float}
\usepackage{fancyhdr}
\setpapersize{A4}
\setmarginsrb{3.0cm}{1.5cm}{3.0cm}{3cm}{1cm}{.5cm}{2cm}{2cm}

\usepackage{amsfonts}
\usepackage{amsthm}
\usepackage[numbers]{natbib}
\usepackage[latin1]{inputenc}
\usepackage[T1]{fontenc}
\usepackage{enumerate}
\usepackage[american]{babel}
\usepackage{graphicx}
\usepackage{bbm}
\usepackage[mathscr]{euscript}
\usepackage{mathrsfs}
\usepackage{stmaryrd}
\usepackage{soul}
\usepackage{hyperref}
\usepackage{xspace}
\usepackage[draft,danish]{fixme}
\usepackage{mathtools}
\usepackage{dsfont}
\usepackage{url}
\usepackage[ps,all,dvips]{xy}
\SelectTips{cm}{12}
\CompileMatrices

  \makeatletter
 \useshorthands{"}
 \defineshorthand{"-}{\nobreak-\bbl@allowhyphens}
 \makeatother

\theoremstyle{plain}
\newtheorem{lemma}{Lemma}[section]
\newtheorem{proposition}[lemma]{Proposition}
\newtheorem{theorem}[lemma]{Theorem}
\newtheorem{corollary}[lemma]{Corollary}
\theoremstyle{definition}
\newtheorem{example}[lemma]{Example}

\newtheorem{condition}[lemma]{Condition}

\theoremstyle{definition}

\theoremstyle{remark}
\newtheorem{remark}[lemma]{Remark}

\newcommand{\devnull}[1]{}

\numberwithin{equation}{section}

\title{On non-stationary solutions to MSDDEs: representations and the cointegration space}
\author{Mikkel Slot Nielsen}
\date{February 25, 2019}
\date{\small Department of Mathematics \\ Aarhus University\\
 mikkel@math.au.dk}
	
\begin{document}
\maketitle

\begin{abstract}
In this paper we study solutions to multivariate stochastic delay differential equations (MSDDEs) which have stationary increments, and we show that this modeling framework is in many ways similar to the discrete-time cointegrated VAR model. In particular, we observe that an MSDDE can always be written in an error correction form and, under suitable conditions, we argue that a process with stationary increments is a solution to the MSDDE if and only if it admits a certain Granger type representation. As a direct implication of these results we obtain a complete characterization of the set of cointegration vectors (the cointegration space). Finally, we exploit the relation between MSDDEs and invertible multivariate CARMA equations to define cointegrated MCARMA processes, and we discuss how this definition is related to earlier literature.
\end{abstract}


\section{Introduction and main results}\label{intro}
Cointegration refers to the phenomenon that some linear combinations of non-stationary time series are stationary. This concept goes at least back to \citet{engle1987co} who used the notion of cointegration to formalize the idea of a long run equilibrium between two or more non-stationary time series. Several models have been shown to be able to embed this idea, and one of the most popular among them is the VAR model:
\begin{equation}\label{cointegratedVAR}
X_t = \Gamma_1 X_{t-1}+ \Gamma_2 X_{t-2} + \cdots + \Gamma_p X_{t-p} + \varepsilon_t,\quad t \in \mathbb{Z}.
\end{equation}
Here $(\varepsilon_t)_{t\in \mathbb{Z}}$ is an $n$-dimensional, say, i.i.d.\ sequence with $\mathbb{E}\varepsilon_0 = 0$ and $\mathbb{E}[\varepsilon_0 \varepsilon_0^T]$ invertible, and $\Gamma_1,\dots, \Gamma_p \in \mathbb{R}^{n\times n}$ are $n\times n$ matrices. If one is searching for a solution $(X_t)_{t\in\mathbb{Z}}$ which is only stationary in its differences, $\Delta X_t := X_t-X_{t-1}$, one often rephrases \eqref{cointegratedVAR} in error correction form
\begin{equation}\label{ecf_var}
\Delta X_t = \Pi_0 X_{t-1} + \sum_{j=1}^{p-1} \Pi_j \Delta X_{t-j} + \varepsilon_t,\quad t \in \mathbb{Z},
\end{equation}
where $\Pi_0 = -I_n + \sum_{j=1}^p \Gamma_j$ and $\Pi_j = -\sum_{k=j+1}^p \Gamma_k$. (Here $I_n$ denotes the $n\times n$ identity matrix.) Properties of solutions to \eqref{cointegratedVAR} concerning existence, uniqueness and stationarity are determined by the characteristic polynomial $\Gamma (z) := I_n-\sum_{j=1}^p \Gamma_j z^j$. Let $r$ be the rank of $\Pi_0 = -\Gamma (1)$ and, if $r<n$, let $\alpha^\perp,\beta^\perp\in \mathbb{R}^{n\times (n-r)}$ be matrices of rank $n-r$ satisfying $\Pi_0^T\alpha^\perp = \Pi_0\beta^\perp = 0$. Standard existence and uniqueness results for VAR models and the Granger representation theorem yield the following:
\begin{theorem}\label{discreteVAR} Suppose that $\det \Gamma (z)= 0$ implies $\vert z \vert > 1$ or $z= 1$. Moreover, suppose either that $r=n$, or $r<n$ and $(\alpha^\perp)^T\bigl(I_n-\sum_{j=1}^{p-1}\Pi_j \bigr)\beta^\perp$ is invertible. Then a process $(X_t)_{t\in \mathbb{Z}}$ with $\mathbb{E}\lVert X_t\rVert^2 <\infty$ and stationary differences is a solution to \eqref{cointegratedVAR} if and only if
\begin{equation}\label{Granger_discrete}
X_t = \xi + C_0 \sum_{j=1}^t \varepsilon_j + \sum_{j=-\infty}^t C (t-j) \varepsilon_j,\quad t \in \mathbb{Z},
\end{equation}
where:
\begin{enumerate}[(i)]
	\item $\xi$ is a random vector satisfying $\mathbb{E}\lVert \xi \rVert^2<\infty$ and $\Pi_0\xi =0$.
	\item $C_0 = \begin{cases} 0 & \text{if } r=n\\
	\beta^\perp \bigl[(\alpha^\perp)^T\bigl(I_n-{\textstyle\sum_{j=1}^{p-1}\Pi_j} \bigr)\beta^\perp\bigr]^{-1} (\alpha^\perp)^T & \text{if } r<n
	\end{cases}$.
	\item $C(j)$ is the $j$-th coefficient in the Taylor expansion of $z\mapsto \Gamma (z)^{-1}-(1-z)^{-1}C_0$ at $z=1$ for $j\geq 0$.
\end{enumerate}
\end{theorem}
(We use the conventions $\sum_{j=1}^0 = 0$ and $\sum_{j=1}^t = -\sum_{j=t+1}^0$ when $t<0$, and $\lVert \cdot \rVert$ denotes the Euclidean norm on $\mathbb{R}^n$.) The representation \eqref{Granger_discrete} has several immediate consequences: (i) any solution with stationary differences can be decomposed into an initial value, a unique stationary part and a unique non-stationary part, (ii) if $r=n$ the solution is stationary and unique, and (iii) if $r<n$ the process $(\gamma^T X_t)_{t\in \mathbb{Z}}$ is stationary if and only if $\gamma\in \mathbb{R}^n$ belongs to the row space of $\Pi_0=- \Gamma(1)$. In particular, cointegration is present in the VAR model when $\Pi_0$ has rank $r\in (0,n)$, and the cointegration space is spanned by the rows of $\Pi_0$.
There exists a massive literature on (cointegrated) VAR models, which have been applied in various fields. We refer to \cite{engle1987co,hansen2005granger,johansen1991estimation,johansen2009cointegration,runkle1987vector,sims1980macroeconomics} for further details.

In many ways, the multivariate stochastic delay differential equation (MSDDE)
\begin{equation}\label{MSDDE_intro}
X_t-X_s = \int_s^t\eta \ast X(u)\, du + Z_t-Z_s,\quad s<t,
\end{equation}
may be viewed as a continuous-time version of the (possibly infinite order) VAR equation \eqref{cointegratedVAR}. Here $Z_t = (Z^1_t,\dots, Z^n_t)^T$, $t\in \mathbb{R}$, is a L\'{e}vy process with $Z_0 = 0$ and $\mathbb{E}\lVert Z_1\rVert^2<\infty$, $\eta$ is an $n\times n$ matrix such that each entry $\eta_{ij}$ is a signed measure on $[0,\infty)$ satisfying
\begin{equation}\label{exponentialDecay_intro}
\int_{[0,\infty)} e^{\delta t}\, \vert \eta_{ij}\vert (dt)<\infty
\end{equation}
for some $\delta>0$, and $\ast$ denotes convolution. (For more on the notation used in this paper, see Section~\ref{prel}.) Moreover, $(X_t)_{t\in \mathbb{R}}$ will be required to satisfy $\mathbb{E}\lVert X_t\rVert^2<\infty$ and be given such that $(X_t,Z_t)_{t\in \mathbb{R}}$ has stationary increments. The precise meaning of \eqref{MSDDE_intro} is that
\begin{equation*}
X^i_t - X^i_s = \sum_{j=1}^n \int_s^t \int_{[0,\infty)} X_{u-v}^j\, \eta_{ij}(dv)\, du + Z_t^i - Z_s^i,\quad i=1,\dots, n,
\end{equation*} 
almost surely for any $s<t$. The model \eqref{MSDDE_intro} results in the multivariate Ornstein-Uhlenbeck process when choosing $\eta = A \delta_0$, $\delta_0$ being the Dirac measure at $0$ and $A\in \mathbb{R}^{n\times n}$, and stationary invertible multivariate CARMA (MCARMA) processes with a non-trivial moving average component can be represented as an MSDDE with infinite delay. Stationary solutions to equations of the type \eqref{MSDDE_intro}, MCARMA processes and their relations have been studied in \cite{basse2018multivariate,brockwell2014recent,GK,marquardt2007multivariate,mohammed1990lyapunov}. Similarly to $\Gamma$ for the VAR model, questions concerning solutions to \eqref{MSDDE_intro} are tied to the function
\begin{equation}\label{hFunction_intro}
h_\eta (z) = zI_n-\int_{[0,\infty)} e^{-zt}\, \eta (dt),\quad \text{Re}(z)>-\delta.
\end{equation}
In particular, it was shown that if $\det h_\eta (z)= 0$ implies $\text{Re}(z)<0$, then the unique stationary solution $(X_t)_{t\in \mathbb{R}}$ to \eqref{MSDDE_intro} with $\mathbb{E}\lVert X_t \rVert^2<\infty$ takes the form 
\begin{equation*}
X_t = \int_{-\infty}^t C(t-u)\, dZ_u,\quad t \in \mathbb{R},
\end{equation*}
where $C:[0,\infty)\to \mathbb{R}^{n\times n}$ is characterized by its Laplace transform:
\begin{equation*}
\int_0^\infty e^{-zt}C(t)\, dt = h_\eta (z)^{-1},\quad \text{Re}(z)\geq 0.
\end{equation*}
It follows that this result is an analogue to Theorem~\ref{discreteVAR} when $r=n$. To the best of our knowledge, there is no literature on solutions to \eqref{MSDDE_intro} which are non-stationary, and hence no counterpart to Theorem~\ref{discreteVAR} exists for the case $r<n$.

\textbf{The main result} of this paper is a complete analogue of Theorem~\ref{discreteVAR}. In the following we will set
\begin{equation}\label{PiQuantities}
\Pi_0 = \eta ([0,\infty))\quad \text{and}\quad \pi (t) = \eta ([0,t]) - \eta ([0,\infty)),\quad t\geq 0.
\end{equation}
In Proposition~\ref{ecf_result_intro} we show that \eqref{MSDDE_intro} admits the following error correction form:
\begin{equation}\label{errorCorrectionForm_intro}
X_t -X_s = \Pi_0\int_s^t X_u\, du + \int_0^\infty \pi (u) (X_{t-u}-X_{s-u})\, du + Z_t-Z_s,\quad s<t.
\end{equation}
To make \eqref{errorCorrectionForm_intro} comparable to \eqref{ecf_var}, one can formally apply the derivative operator $D$ to the equation and obtain
\begin{equation}\label{derivative_ecf}
DX_t = \Pi_0 X_t + \Pi \ast (DX)(t) + DZ_t,\quad t \in \mathbb{R},
\end{equation}
with $\Pi (dt) = \pi (t)\, dt$. We can now formulate the counterpart to Theorem~\ref{discreteVAR}. In the following, $r$ refers to the rank of $\Pi_0$ and in case $r<n$, then $\alpha^\perp,\beta^\perp\in \mathbb{R}^{n\times (n-r)}$ are matrices of rank $n-r$ which satisfy $\Pi_0^T\alpha^\perp = \Pi_0\beta^\perp = 0$.

\begin{theorem}\label{Granger_intro}
	Suppose that $\det h_\eta (z) =0$ implies $\text{Re}(z)<0$ or $z=0$. Moreover, suppose either that the rank $r$ of $\Pi_0$ is $n$, or strictly less than $n$ and $(\alpha^\perp)^T(I_n-\Pi ([0,\infty)))\beta^\perp$ is invertible. Then a process $(X_t)_{t\in \mathbb{R}}$ is a solution to \eqref{MSDDE_intro} if and only if
	\begin{equation}\label{Granger_form_intro}
	X_t = \xi + C_0 Z_t + \int_{-\infty}^t C(t-u)\, dZ_u,\quad t \in \mathbb{R},
	\end{equation}
	where the following holds:
	\begin{enumerate}[(i)]
		\item $\xi$ is a random vector satisfying $\mathbb{E}\lVert \xi \rVert^2<\infty$ and $\Pi_0 \xi =0$.
		\item $C_0 = \begin{cases}
		0 & \text{if }r=n \\
		\beta^\perp \bigl[(\alpha^\perp)^T(I_n - \Pi ([0,\infty)))\beta^\perp \bigr]^{-1}(\alpha^\perp)^T & \text{if } r<n
		\end{cases}$.
		\item $C:[0,\infty) \to \mathbb{R}^{n\times n}$ is characterized by 
		\begin{equation*}
		\int_0^\infty e^{-zt}C(t)\, dt = h_\eta (z)^{-1}-z^{-1}C_0,\quad \text{Re}(z)\geq 0.
		\end{equation*}
	\end{enumerate}
\end{theorem}
Similarly to the VAR model, Theorem~\ref{Granger_intro} shows that cointegration occurs in the MSDDE model when $\Pi_0$ is of reduced rank $r\in (0,n)$, and the rows of $\Pi_0$ span the cointegration space. It follows as well that we always have uniqueness up to the discrepancy term $\xi$, and the restrictions on $\xi$ depend ultimately on the rank of $\Pi_0$. Since an invertible MCARMA equation may be rephrased as an MSDDE, the notion of cointegrated invertible MCARMA processes can be studied in the MSDDE framework by relying on Theorem~\ref{Granger_intro} (see Section~\ref{CARMAsection} for details).

In Section~\ref{prel} we will introduce some notation which will be used throughout the paper, and which already has been used in the introduction. The purpose of Section~\ref{existenceSection} is to develop a general theory for non-stationary solutions to MSDDEs with stationary increments, some of which will later be used to prove Theorem~\ref{Granger_intro}. In this section we will also put some emphasis on the implications of the representation \eqref{Granger_form_intro}, both in terms of stationary properties and concrete examples. Section~\ref{CARMAsection} discusses how one can rely on the relation between invertible MCARMA equations and MSDDEs to define cointegrated MCARMA processes. In particular, under conditions similar to those imposed in \cite[Theorem~4.6]{fasen2016cointegrated}, we show existence and uniqueness of a cointegrated solution to the MSDDE associated to the MCARMA($p,p-1$) equation. This complements the result of \cite{fasen2016cointegrated}, which ensures existence of cointegrated MCARMA($p,q$) processes when $p>q+1$. Finally, Section~\ref{proofs} contains the proofs of all the statements presented in the paper together with a few technical results.
\section{Preliminaries}\label{prel}
Let $f=(f_{ij}):\mathbb{R}\to \mathbb{C}^{m\times k}$ be a measurable function and $\mu = (\mu_{ij})$ a $k\times n$ matrix where each $\mu_{ij}$ is a measure on $\mathbb{R}$. Then, provided that
\begin{equation*}
\int_\mathbb{R} \vert f_{il}(t)\vert\, \mu_{lj}(dt)<\infty
\end{equation*}
for $l=1,\dots, k$, $i=1,\dots, m$ and $j=1,\dots, n$, we set
\begin{equation}\label{integration}
\int_\mathbb{R} f(t)\, \mu (dt) = \sum_{l=1}^k \begin{bmatrix}
\int_\mathbb{R} f_{1l}(t)\, \mu_{l1}(dt) & \cdots & \int_\mathbb{R} f_{1l}(t)\, \mu_{ln} (dt) \\
\vdots & \ddots & \vdots \\
\int_\mathbb{R} f_{ml}(t)\, \mu_{l1}(dt) & \cdots & \int_\mathbb{R}f_{ml}(t)\, \mu_{ln}(dt)
\end{bmatrix}
\end{equation}
The definition of $\int_\mathbb{R}f(t)\, dt$ is defined in a similar (obvious) manner when either $f$ or $\mu$ is one-dimensional. Moreover, we will say that $\mu$ is a signed measure if it takes the form $\mu = \mu^+ - \mu^-$ for two mutually singular measures $\mu^+$ and $\mu^-$, where at least one of them is finite. The definition of the integral \eqref{integration} extends naturally to signed matrix measures provided that the integrand is integrable with respect to the variation measure $\vert \mu \vert:=\mu^+ + \mu^-$ (simply referred to as being integrable with respect to $\mu$). For a given point $t\in \mathbb{R}$, if $f(t-\cdot)$ is integrable with respect to $\mu$, we define the convolution as
\begin{equation*}
f\ast \mu (t) = \int_\mathbb{R} f(t-u)\, \mu (du).
\end{equation*} 
For a measurable function $f:\mathbb{R}\to \mathbb{C}^{k\times m}$ and $\mu$ an $n\times k$ signed matrix measure, if $f^T(t-\cdot)$ is integrable with respect to $\mu^T$, we set $\mu \ast f (t) := (f^T \ast \mu^T)(t)^T$. Also, if $\mu$ is a given signed matrix measure and $z\in \mathbb{C}$ is such that $\int_\mathbb{R} e^{-\text{Re}(z)t}\, \vert \mu_{ij}\vert (dt)<\infty$ for all $i$ and $j$, the $(i,j)$-th entry of the Laplace transform $\mathcal{L}[\mu](z)$ of $\mu$ at $z$ is defined by
\begin{equation*}
\mathcal{L}[\mu]_{ij}(z) = \int_\mathbb{R} e^{-zt}\, \mu_{ij} (dt).
\end{equation*}
Eventually, if $\vert \mu\vert$ is finite, we will also use the notation $\mathcal{F}[\mu](y) = \mathcal{L}[\mu](iy)$, $y\in \mathbb{R}$, referring to the Fourier transform of $\mu$. When $\mu (dt) = f(t)\, dt$ for some measurable function $f$ we write $\mathcal{L}[f]$ and $\mathcal{F}[f]$ instead.

Finally, a stochastic process $Y_t = (Y_t^1,\dots, Y_t^n)^T$, $t\in \mathbb{R}$, is said to be stationary, respectively have stationary increments, if the finite dimensional marginal distributions of $(Y_{t+h})_{t\in \mathbb{R}}$, respectively $(Y_{t+h}-Y_h)_{t\in \mathbb{R}}$, do not depend on $h\in \mathbb{R}$. 

\section{General results on existence, uniqueness and representations of solutions to MSDDEs}\label{existenceSection}
Suppose that $Z_t = (Z_t^1,\dots, Z_t^n)^T$, $t\in\mathbb{R}$, is an $n$-dimensional measurable process with $Z_0 =0$, stationary increments and $\mathbb{E}\lVert Z_t \rVert^2 <\infty$, and let $\eta = (\eta_{ij})$ be a signed $n\times n$ matrix measure which satisfies \eqref{exponentialDecay_intro} for some $\delta>0$. We will say that a stochastic process $X_t = (X^1_t,\dots, X^n_t)^T$, $t\in \mathbb{R}$, is a solution to the corresponding multivariate stochastic delay differential equation (MSDDE) if it meets the following requirements:
\begin{enumerate}[(i)]
	\item $(X_t)_{t\in \mathbb{R}}$ is measurable and $\mathbb{E}\lVert X_t \rVert^2 < \infty$ for all $t\in \mathbb{R}$.
 	\item $(X_t,Z_t)_{t\in \mathbb{R}}$ has stationary increments.
 	\item\label{solutionEq} The relations
 	\begin{equation*}
 	X^i_t - X^i_s = \sum_{j=1}^n \int_s^t \int_{[0,\infty)} X_{u-v}^j\, \eta_{ij}(dv)\, du + Z_t^i - Z_s^i,\quad i=1,\dots, n,
 	\end{equation*} 
 	hold true almost surely for each $s<t$.
\end{enumerate}
As indicated in the introduction, \eqref{solutionEq} may be compactly written as
\begin{equation}\label{MSDDE}
dX_t = \eta \ast X(t)\, dt + dZ_t,\quad t \in \mathbb{R}.
\end{equation}
We start with the observation that \eqref{MSDDE} can always be written in an error correction form (as noted in \eqref{errorCorrectionForm_intro}):
\begin{proposition}\label{ecf_result_intro}
	Let $\Pi_0\in \mathbb{R}^{n\times n}$ and $\pi:[0,\infty) \to \mathbb{R}^{n\times n}$ be defined by \eqref{PiQuantities}, and suppose that $\delta>0$ is given such that \eqref{exponentialDecay_intro} is satisfied. Then $\sup_{t\geq 0} e^{\varepsilon t}\lVert \pi (t) \rVert <\infty$ for all $\varepsilon < \delta$, and \eqref{MSDDE} can be written as
	\begin{equation}\label{errorCorrectionForm_main}
	X_t -X_s = \Pi_0\int_s^t X_u\, du + \int_0^\infty \pi (u) (X_{t-u}-X_{s-u})\, du + Z_t-Z_s,\quad s<t,
	\end{equation}
	so if $(X_t)_{t\in \mathbb{R}}$ is a solution to \eqref{MSDDE}, then $(\Pi_0 X_t)_{t\in \mathbb{R}}$ is stationary.
\end{proposition}

\begin{remark}\label{stationarityHeuristics}
Using the notation $\Pi (dt) = \pi (t)\, dt$, we do the following observations in relation to Proposition~\ref{ecf_result_intro}:
\begin{enumerate}[(i)]
	\item If $\Pi_0$ is invertible, a solution $(X_t)_{t\in \mathbb{R}}$ must be stationary itself. 
	\item If $\Pi_0 = 0$ the statement does not provide any further insight. Observe, however, the equation \eqref{errorCorrectionForm_main} depends in this case only on the increments of $(X_t)_{t\in \mathbb{R}}$ so a solution needs not to be stationary in this case.
	\item\label{lastCase} If the rank $r$ of $\Pi_0$ satisfies $0<r<n$, there exist non-trivial linear combinations of the entries of $(X_t)_{t\in \mathbb{R}}$ which are stationary
\end{enumerate}
At this point we have not argued whether or not $(X_t)_{t\in \mathbb{R}}$ can be stationary even when $r<n$ and, ultimately, it depends on the structure of the noise process $(Z_t)_{t\in \mathbb{R}}$. However, it is not too difficult to verify from Theorem~\ref{existence} that if $(Z_t)_{t\in \mathbb{R}}$ is a L\'{e}vy process such that $\mathbb{E}[Z_1Z_1^T]$ is invertible and $A\in \mathbb{R}^{m\times n}$, $(AX_t)_{t\in \mathbb{R}}$ is stationary if and only if $A = B\Pi_0$ for some $B\in \mathbb{R}^{m\times n}$. In case of \eqref{lastCase}, one often considers a rank factorization of $\Pi_0$; that is, one chooses $\alpha,\beta\in \mathbb{R}^{n \times r}$ of rank $r$ such that $\Pi_0 = \alpha \beta^T$. In this way one can identify the columns of $\beta$ as cointegrating vectors spanning the cointegration space, and $\alpha$ as the adjustment matrix determining how deviations from a long run equilibrium affect short run dynamics. This type of intuition is well-known for the cointegrated VAR models, so we refer to \cite{engle1987co} for details.
\end{remark}

In the following we will search for a solution to \eqref{MSDDE}. To this end, let $\delta>0$ be chosen such that \eqref{exponentialDecay_intro} holds, set $\mathbb{H}_\delta := \{z\in \mathbb{C}\, :\, \text{Re}(z)>-\delta\}$ and define $h_\eta: \mathbb{H}_\delta \to \mathbb{C}^{n\times n}$ by 
\begin{equation}\label{hFunction}
h_\eta (z) = zI_n - \mathcal{L}[\eta](z),\quad z \in \mathbb{H}_\delta.
\end{equation}
Since $h_\eta$ is analytic on $\mathbb{H}_\delta$ and $\vert \det h_\eta (z)\vert \to \infty$ as $\vert z\vert \to \infty$, $z\mapsto h_\eta (z)^{-1}$ is meromorphic on $\mathbb{H}_\delta$. Recall that if $z_0$ is a pole of $z\mapsto h_\eta (z)^{-1}$, there exists $n \in \mathbb{N}$ such that $z\mapsto (z-z_0)^n h_\eta (z)^{-1}$ is analytic and non-zero in a neighborhood of $z_0$. If $n=1$ the pole is called simple.

\begin{condition}\label{mostIone} For the function $h_\eta$ in \eqref{hFunction} it holds that
	\begin{enumerate}[(i)]
		\item\label{cond1} $\det (h_\eta (z))\neq 0$ for all $z\in \mathbb{H}_\delta \setminus\{0\}$ and
		\item\label{cond2} $z\mapsto h_\eta (z)^{-1}$ has either no poles at all or a simple pole at $0$.
	\end{enumerate}
	
\end{condition}
For convenience, we have chosen to work with Condition~\ref{mostIone} rather than the assumptions of Theorem~\ref{Granger_intro}. The following result shows that they are essentially the same.
\begin{proposition}\label{propEqOfAssump}
	Suppose that, for some $\varepsilon >0$,
	\begin{equation*}
	\int_{[0,\infty)} e^{\varepsilon t}\, \vert \eta_{ij}\vert (dt)<\infty,\quad i,j = 1,\dots, n.
	\end{equation*}
	The following two statements are equivalent:
	\begin{enumerate}[(i)]
		\item\label{firstProp} There exists $\delta \in (0,\varepsilon]$ such that \eqref{exponentialDecay_intro} and Condition~\ref{mostIone} are satisfied.
		\item\label{secondProp} The assumptions of Theorem~\ref{Granger_intro} hold true.
	\end{enumerate}
\end{proposition}

We will construct a solution $(X_t)_{t\in \mathbb{R}}$ to \eqref{MSDDE} in a similar way as in \cite{basse2018multivariate}, namely by applying a suitable filter (i.e., a finite signed $n\times n$ matrix measure) $\mu$ to $(Z_t)_{t\in \mathbb{R}}$. Theorem~\ref{existence} reveals that the appropriate filter to apply is $\mu (du) = \delta_0(du) - f(u)\, du$ for a suitable function $f:\mathbb{R}\to \mathbb{R}^{n\times n}$. This result may be viewed as a Granger type representation theorem for solutions to MSDDEs and as a general version of Theorem~\ref{Granger_intro}.

\begin{theorem}\label{existence}
	Suppose that Condition~\ref{mostIone} holds. Then there exists a unique function $f:[0,\infty)\to \mathbb{R}^{n\times n}$ satisfying
	\begin{equation}\label{laplaceChar}
	\mathcal{L}[f](z) = I_n - zh_\eta (z)^{-1},\quad z \in \mathbb{H}_\delta,
	\end{equation}
	and the function $u \mapsto f(u)Z_{t-u}$ belongs to $L^1$ almost surely for each $t\in \mathbb{R}$. Moreover, a process $(X_t)_{t\in \mathbb{R}}$ is a solution to \eqref{MSDDE} if and only if
	\begin{equation}\label{solution}
	X_t = \xi + C_0Z_t + \int_0^\infty f(u) [Z_t-Z_{t-u}]\, du,\quad t \in \mathbb{R},
	\end{equation}
	where $\Pi_0 \xi = 0$, $\mathbb{E}\lVert\xi \rVert^2<\infty$ and $C_0 = I_n - \int_0^\infty f(t)\, dt$. 
\end{theorem}
Concerning the function $f$ of Theorem~\ref{existence}, it can also be obtained as a solution to a certain multivariate delay differential equation; we refer to Lemma~\ref{fExistence} for more on its properties.  

\begin{remark}\label{C0_structure}
	Let the situation be as described in Theorem~\ref{existence} and note that
	\begin{equation*}
	C_0 = I_n-\mathcal{L}[f](0) = zh_\eta (z)^{-1}\vert_{z=0}.
	\end{equation*}
	 Hence, if the rank $r$ of $\Pi_0$ is equal to $n$ we have that $C_0=0$, and if $r$ is strictly less than $n$, $C_0$ can be computed by the residue formula given in \cite{schumacher1991system}. Specifically, $C_0 = \beta^\perp [(\alpha^\perp)^T(I_n-\Pi ([0,\infty)))\beta^\perp]^{-1}(\alpha^\perp)^T$, where $\alpha^\perp,\beta^\perp\in \mathbb{R}^{n\times (n-r)}$ are matrices of rank $n-r$ satisfying $\Pi_0^T\alpha^\perp = \Pi_0 \beta^\perp = 0$ (note that the inverse matrix in the expression of $C_0$ does indeed exist by Proposition~\ref{propEqOfAssump}).
\end{remark}
In the special case where $z\mapsto h_\eta (z)^{-1}$ has no poles at all, it was shown in \cite[Theorem~3.1]{basse2018multivariate} that there exists a unique stationary solution to \eqref{MSDDE}. The same conclusion can be reached by Theorem~\ref{existence} using that $\Pi_0$ is invertible. Indeed, in this case any solution is stationary, $C_0 = 0$ and $\xi = 0$ (the first two implications follow from Remark~\ref{stationarityHeuristics} and \ref{C0_structure}, respectively). While there exist several solutions when $\Pi_0$ is singular, Theorem~\ref{existence} any two solutions always have the same increments. The term $\xi$ reflects how much solutions may differ and its possible values are determined by the relation $\Pi_0 \xi = 0$.
In view of Proposition~\ref{propEqOfAssump} and Remark~\ref{C0_structure}, Theorem~\ref{Granger_intro} is an obvious consequence of Theorem~\ref{existence} if
\begin{equation}\label{GrangerLack}
\int_0^\infty f(u)[Z_t-Z_{t-u}]\, dt  = \int_{-\infty}^t C(t-u)\, dZ_u,\quad t \in \mathbb{R}.
\end{equation}
Clearly, the right-hand side of \eqref{GrangerLack} requires that we can define integration with respect to $(Z_t)_{t\in \mathbb{R}}$. Although this is indeed possible if $(Z_t)_{t\in \mathbb{R}}$ is a L\'{e}vy process (for instance, in the sense of \cite{rosSpec}), we will here put the less restrictive assumption that $(Z_t)_{t\in \mathbb{R}}$ is a regular integrator as defined in \cite[Proposition~4.1]{basse2018multivariate}:
\begin{corollary}\label{regGranger}
	Suppose that Condition~\ref{mostIone} holds. Assume also that, for each $i = 1,\dots, n$, there exists a linear map $\mathcal{I}_i:L^1 \cap L^2 \to L^1(\mathbb{P})$ which satisfies
	\begin{enumerate}[(i)]
		\item\label{reg1} $\mathcal{I}_i (\mathds{1}_{(s,t]}) = Z^i_t-Z^i_s$ for all $s<t$, and
		\item\label{reg2} for all finite measures $\mu$ on $\mathbb{R}$ with $\int_\mathbb{R}\vert r\vert \, \mu (dr) < \infty$,
		\begin{equation*}
		\mathcal{I}_i \Bigl(\int_\mathbb{R} f_r(t-\cdot)\, \mu (dr)\Bigr) = \int_\mathbb{R} \mathcal{I}_i (f_r(t-\cdot))\, \mu (dr),\quad t\in \mathbb{R},
		\end{equation*}
		where $f_r = \mathds{1}_{[0,\infty)}(\cdot - r)-\mathds{1}_{[0,\infty)}$.
	\end{enumerate}
Then the statement of Theorem~\ref{Granger_intro} holds true with
\begin{equation}\label{integralDefine}
\Bigl(\int_{-\infty}^t C(t-u)\, dZ_u\Bigr)_i = \sum_{j=1}^n \mathcal{I}_j (C_{ij}(t-\cdot)),\quad i=1,\dots, n.
\end{equation}
\end{corollary}
In Theorem~\ref{Granger_intro} the function $C$ is characterized through its Laplace transform $\mathcal{L}[C]$, but one can also obtain it as a solution to a certain multivariate delay differential equation. This follows by using the similar characterization given for $f$ in Lemma~\ref{fExistence}; the details are discussed in Remark~\ref{Cdyn}. It should also be stressed that the conditions for being a regular integrator (i.e., for $\mathcal{I}_1,\dots, \mathcal{I}_n$ to exist) are mild; many semimartingales with stationary increments (in particular, L\'{e}vy processes) and fractional L\'{e}vy processes, as studied in \cite{Tina}, are regular integrators. For more on regular integrators, see \cite[Section~4.1]{basse2018multivariate}.
\begin{remark}
Suppose that Condition~\ref{mostIone} is satisfied, let $(Z_t)_{t\in \mathbb{R}}$ be a regular integrator, and let $(X_t)_{t\in \mathbb{R}}$ be a solution to \eqref{MSDDE}. Since $\Pi_0 \xi = 0$ and $\Pi_0 C_0 = 0$ (the latter by Remark~\ref{C0_structure}), Corollary~\ref{regGranger} implies that the stationary process $(\Pi_0 X_t)_{t\in \mathbb{R}}$ is unique and given by
\begin{equation}\label{uniqueStatComp}
\Pi_0 X_t = \Pi_0\int_{-\infty}^t C(t-u)\, dZ_u,\quad t \in \mathbb{R}.
\end{equation}
If $(Z_t)_{t\in \mathbb{R}}$ is not a regular integrator one can instead rely on Theorem~\ref{existence} to replace $\int_{-\infty}^tC(t-u)\, dZ_u$ by $\int_0^\infty f(u) [Z_t-Z_{t-u}]\, du$ in \eqref{uniqueStatComp}.
\end{remark}
We end this section by giving two examples. In both examples we suppose for convenience that $(Z_t)_{t\in \mathbb{R}}$ is a regular integrator.

\begin{example}[The univariate case]
	Consider the case where $n=1$ and $\eta$ is a measure which admits an exponential moment in the sense of \eqref{exponentialDecay_intro} and satisfies $h_\eta (z)\neq 0$ for all $z\in \mathbb{H}_\delta\setminus \{0\}$. In this setup Condition~\ref{mostIone} can be satisfied in two ways which ultimately determine the class of solutions characterized in Corollary~\ref{regGranger}:
	\begin{enumerate}[(i)]
		\item If $\Pi_0\neq 0$. In this case, the solution to \eqref{MSDDE} is unique and given by
		\begin{equation*}
		X_t = \int_{-\infty}^t C(t-u)\, dZ_u,\quad t \in \mathbb{R},
		\end{equation*}
		where $\mathcal{L}[C](z) = 1/h_\eta (z)$ for $z\in \mathbb{C}$ with $\text{Re}(z)\geq 0$. This is consistent with the literature on stationary solutions to univariate SDDEs (see \cite{contARMA,GK}).
		
		\item\label{nonStatUniv} If $\Pi_0 = 0$ and $\Pi ([0,\infty)) \neq 1$. In this case, a process $(X_t)_{t\in \mathbb{R}}$ is a solution to \eqref{MSDDE} if and only if
		\begin{equation*}
		X_t = \xi + \bigl(1-\Pi ([0,\infty))\bigr)Z_t + \int_{-\infty}^t C(t-u)\, dZ_u,\quad t \in \mathbb{R},
		\end{equation*}
		where $\xi$ can be any random variable with $\mathbb{E}\xi^2<\infty$ and $\mathcal{L}[C](z) = 1/h_\eta (z) - (1-\Pi ([0,\infty)))/z$ for $z\in \mathbb{C}$ with $\text{Re}(z)\geq 0$.
	\end{enumerate}
	Suppose that we are in case \eqref{nonStatUniv} and fix $h>0$. Using the notation $\Delta_h Y_t := Y_t-Y_{t-h}$, it follows from Proposition~\ref{ecf_result_intro} that $(\Delta_h X_t)_{t\in \mathbb{R}}$ is a stationary solution to the equation
	\begin{equation}\label{levelEq}
	Y_t = \int_0^\infty Y_{t-u}\, \Pi (du) + \Delta_h Z_t,\quad t\in \mathbb{R}.
	\end{equation}
	Existence and uniqueness of stationary solutions to equations of the type \eqref{levelEq} were studied in \cite[Section~3]{contARMA} (when $(\Delta_hZ_t)_{t\in \mathbb{R}}$ is a suitable L\'{e}vy-driven moving average), and it was shown how these sometimes can be used to construct stationary increment solutions to univariate SDDEs.
\end{example}

\begin{example}[Ornstein-Uhlenbeck]
	Suppose that $\eta = A \delta_0$ for some $A\in \mathbb{R}^{n\times n}$, for which its spectrum $\sigma (A)$ satisfies
	\begin{equation}\label{aSpectrum}
	\sigma (A)\setminus \{0\}\subseteq \{z\in \mathbb{C}\, :\, \text{Re}(z)<0\}.
	\end{equation}
	With this specification, the MSDDE \eqref{MSDDE} reads
	\begin{equation}\label{OUeq}
	dX_t = AX_t\, dt + dZ_t,\quad t\in \mathbb{R}.
	\end{equation}
	Under the assumption \eqref{aSpectrum} we have that
	\begin{equation*}
	h_\eta (z)^{-1} = \int_0^\infty e^{(A-I_nz)t}\, dt = \mathcal{L}\bigl[t\mapsto \mathds{1}_{[0,\infty)}(t)e^{At} \bigr] (z),\quad \text{Re}(z)>0.
	\end{equation*}
	Since the set of zeroes of $h_\eta$ coincides with $\sigma (A)$, it follows immediately that Condition~\ref{mostIone} is satisfied for some $\delta>0$ if $0 \notin \sigma (A)$. This is the stationary case where the solution to \eqref{OUeq} takes the well-known form
	\begin{equation*}
	X_t = \int_{-\infty}^t e^{A (t-u)}\, dZ_u,\quad t \in \mathbb{R}.
	\end{equation*}
 	If instead $0 \in \sigma (A)$, let $r<n$ be the rank of $A$ and choose $\alpha^\perp,\beta^\perp \in \mathbb{R}^{n\times (n-r)}$ of rank $n-r$ such that $A^T\alpha^\perp = A \beta^\perp = 0$. We can now rely on Proposition~\ref{propEqOfAssump} and the observation that $\Pi \equiv 0$ to conclude that Condition~\ref{mostIone} is satisfied if $(\alpha^\perp)^T\beta^\perp$ is invertible. This is the cointegrated case where the solution takes the form
 	\begin{equation*}
 	X_t = \xi + C_0Z_t + \int_{-\infty}^t \bigl[e^{A(t-u)}-C_0 \bigr]\, dZ_u,\quad t \in \mathbb{R},
 	\end{equation*}
 	with $\mathbb{E}\lVert \xi \rVert^2<\infty$, $A\xi = 0$ and $C_0 = \beta^\perp \bigl[(\alpha^\perp)^T\beta^\perp \bigr]^{-1} (\alpha^\perp)^T$. In particular, the stationary process $(AX_t)_{t\in \mathbb{R}}$ takes the form
 	\begin{equation*}
 	AX_t = \int_{-\infty}^t Ae^{A(t-u)}\, dZ_u,\quad t \in \mathbb{R}.
 	\end{equation*}
 	 Stationary Ornstein-Uhlenbeck processes have been widely studied in the literature (see, e.g., \cite{barndorff1998some,sato1994recurrence,Sato_OU}). Cointegrated solutions to \eqref{OUeq} have also received some attention, for instance, in \cite{comte1999discrete}.
\end{example}

\section{Cointegrated multivariate CARMA processes}\label{CARMAsection}
In \cite[Theorem~4.8]{basse2018multivariate} it was shown that any stationary MCARMA process satisfying a certain invertibility assumption can be characterized as the unique solution to a suitable MSDDE. This may be viewed as the continuous-time analogue of representing a discrete-time ARMA process as an infinite order AR equation. In this section we will rely on this idea and the results obtained in Section~\ref{existenceSection} to define cointegrated MCARMA processes. The focus will only be on MCARMA($p,p-1$) processes for a given $p\in \mathbb{N}$. However, the analysis should also be doable for MCARMA($p,q$) processes for a general $q\in \mathbb{N}_0$ with $q<p$ by extending the theory developed in the former sections to higher order MSDDEs. This was done in \cite{basse2018multivariate} in the stationary case. For convenience we will also assume that $(Z_t)_{t\in \mathbb{R}}$ is a regular integrator in the sense of Corollary~\ref{regGranger}.

We start by introducing some notation. Define $P,Q:\mathbb{C}\to \mathbb{C}^{n\times n}$ by
	\begin{align*}
	P(z) &= I_n z^p + P_1 z^{p-1} + \cdots + P_p\quad \text{and}\\
	Q(z) &= I_n z^{p-1} + Q_1 z^{p-2} + \cdots + Q_{p-1}
	\end{align*}
for $P_1,\dots, P_p,Q_1,\dots, Q_{p-1}\in \mathbb{R}^{n\times n}$. Essentially, any definition of the MCARMA process $(X_t)_{t\in \mathbb{R}}$ aims at rigorously defining the solution to the formal differential equation
\begin{equation}\label{carmaEq}
P(D)X_t = Q(D)DZ_t,\quad t \in \mathbb{R}.
\end{equation}
Since $P(D)\xi = P_p \xi$ for any random vector $\xi$, one should only expect solutions to be unique up to translations belonging to the null space of $P_p$. To solve \eqref{carmaEq} it is only necessary to impose assumptions on $P$, but since we will be interested in an autoregressive representation of the equation, we will also impose an invertibility assumption on $Q$:
\begin{condition}[Stationary case]\label{statAssump} If $\det P(z) = 0$ or $\det Q(z) = 0$, then $\text{Re}(z)<0$.
\end{condition}

Under Condition~\ref{statAssump} it was noted in \cite[Remark~3.23]{marquardt2007multivariate} that one can find $g:[0,\infty) \to \mathbb{R}^{n\times n}$ which belongs to $L^1 \cap L^2$ with 
\begin{equation}\label{CARMAkernel}
\mathcal{F}[g](y) = P(iy)^{-1}Q(iy),\quad y\in \mathbb{R}.
\end{equation}
Consequently, by heuristically applying the Fourier transform to \eqref{carmaEq} and rearranging terms, one arrives at the conclusion 
\begin{equation}\label{CARMAasMA}
X_t = \int_{-\infty}^t g(t-u)\, dZ_u,\quad t \in\mathbb{R}.
\end{equation}
As should be the case, any definition used in the literature results in this process (although $(Z_t)_{t\in \mathbb{R}}$ is sometimes restricted to being a L\'{e}vy process). In Proposition~\ref{charCARMA} we state two characterizations without proofs; these are consequences of \cite[Definition~3.20]{marquardt2007multivariate} and \cite[Theorem~4.8]{basse2018multivariate}, respectively.

\begin{proposition}\label{charCARMA} Suppose that Condition~\ref{statAssump} is satisfied and let $(X_t)_{t\in \mathbb{R}}$ be defined by \eqref{CARMAkernel}-\eqref{CARMAasMA}.
	\begin{enumerate}[(i)]
		\item\label{stateSpaceRep} Choose $B_1,\dots, B_p\in \mathbb{R}^{n\times n}$ such that $z\mapsto P(z) [B_1z^{p-1}+\cdots + B_p] - Q(z) z^p$ is at most of order $p-1$, and set
		\begin{equation*}
		A = \begin{bmatrix}
		0 & I_n & 0 & \cdots & 0 \\
		0 & 0 & I_n & \cdots & 0 \\
		\vdots & \vdots & \ddots & \ddots & \vdots \\
		0 & 0 & \cdots & 0 & I_n \\
		-P_p & -P_{p-1} & \cdots &-P_2 &-P_1
		\end{bmatrix}\quad \text{and}\quad B = \begin{bmatrix}
		B_1 \\ \vdots \\ B_p
		\end{bmatrix}.
		\end{equation*}
		Then $X_t = C G_t$, where $C = [I_n,0,\dots, 0]^T\in \mathbb{R}^{np\times n}$ and $(G_t)_{t\in \mathbb{R}}$ is the unique stationary process satisfying $dG_t = AG_t\, dt + B\, dZ_t$ for $t \in \mathbb{R}$.
		\item\label{SDDErep} Set $\eta_0=Q_1-P_1$ and let $\eta_1:[0,\infty)\to \mathbb{R}^{n\times n}$ be characterized by 
		\begin{equation}\label{etaOne}
		\mathcal{F}[\eta_1](y) = I_n iy-\eta_0-Q(iy)^{-1}P(iy),\quad y\in \mathbb{R}.
		\end{equation}
		 Then $(X_t)_{t\in \mathbb{R}}$ is the unique stationary process satisfying
		\begin{equation*}
		dX_t = \eta_0 X_t\, dt + \int_0^\infty \eta_1 (u) X_{t-u}\, du\, dt + dZ_t,\quad t \in \mathbb{R}.
		\end{equation*}	 
	\end{enumerate}
\end{proposition}
 
It follows from Proposition~\ref{charCARMA} that $(X_t)_{t\in \mathbb{R}}$ can either be defined in terms of a state-space model using the triple $(A,B,C)$ or by an MSDDE of the form \eqref{MSDDE} with 
\begin{equation}\label{etaForCARMA}
\eta (dt) = \eta_0 \delta_0 (dt) + \eta_1(t)\, dt.
\end{equation}
 While $(X_t)_{t\in \mathbb{R}}$ given by \eqref{CARMAasMA} is stationary by definition, it does indeed make sense to search for non-stationary, but cointegrated, processes satisfying \eqref{stateSpaceRep} or \eqref{SDDErep} of Proposition~\ref{charCARMA} also when Condition~\ref{statAssump} does not hold. \citet{fasen2016cointegrated} follow this idea by first characterizing cointegrated solutions to state-space equations and, next, define the cointegrated MCARMA process as a cointegrated solution corresponding to the specific triple $(A,B,C)$. Their definition applies to any MCARMA($p,q$) process and they give sufficient conditions on $P$ and $Q$ for the cointegrated MCARMA process to exist when $q<p-1$. We will use the results from the former sections to define the cointegrated MCARMA($p,p-1$) process as the solution to an MSDDE.

\begin{condition}[Cointegrated case]\label{cointAssump}
	The following statements are true:
	\begin{enumerate}[(i)]
		\item\label{firstCond} If $\det P (z) =0$, then either $\text{Re}(z)<0$ or $z=0$.
		\item The rank $r$ of $P(0) = P_p$ is reduced $r\in (0,n)$.
		\item\label{midCond} The matrix $(\alpha^\perp)^TP_{p-1} \beta^\perp$ is invertible, where $\alpha^\perp,\beta^\perp \in \mathbb{R}^{n\times (n-r)}$ are of rank $n-r$ and satisfy $P_p^T\alpha^\perp = P_p\beta^\perp=0$.
		\item\label{invAssump} If $\det Q(z) =0$ then $\text{Re} (z) <0$.
	\end{enumerate}
\end{condition}
The assumptions \eqref{firstCond}-\eqref{midCond} of Condition~\ref{cointAssump} are also imposed in \cite{fasen2016cointegrated}, and \eqref{invAssump} is imposed to ensure that \eqref{carmaEq} admits an MSDDE representation. In \cite{fasen2016cointegrated} they impose an additional assumption, namely that the polynomials $P$ and $Q$ are so-called left coprime, which is used to ensure that the pole of $z\mapsto P(z)^{-1}$ at $0$ is also a pole of $z\mapsto P(z)^{-1}Q(z)$. However, in our case this is implied by \eqref{invAssump}.

\begin{theorem}\label{carmaCointTheorem}
	Suppose that Condition~\ref{cointAssump} holds. Then the measure in \eqref{etaForCARMA} is well-defined and satisfies \eqref{exponentialDecay_intro} as well as Condition~\ref{mostIone} for a suitable $\delta>0$, and the rank of $\Pi_0 =\eta ([0,\infty))$ is $r$. In particular, a process $(X_t)_{t\in \mathbb{R}}$ is a solution to the corresponding MSDDE if and only if
	\begin{equation}\label{cointegratedCARMA}
	X_t = \xi  + C_0 Z_t + \int_{-\infty}^t C(t-u)\, dZ_u,\quad t \in \mathbb{R},
	\end{equation}
	where $\mathbb{E}\lVert\xi \rVert^2 <\infty $, $P_p \xi =0$, $C_0 = \beta^\perp \bigl[(\alpha^\perp)^T P_{p-1}\beta^\perp \bigr]^{-1}(\alpha^\perp)^T Q_{p-1}$ and $\mathcal{L}[C](z) = P(z)^{-1}Q(z)-z^{-1}C_0$ for $z\in \mathbb{C}$ with $\text{Re}(z)\geq 0$.
\end{theorem}
\begin{remark}
Suppose that Condition~\ref{cointAssump} is satisfied and define $\eta$ by \eqref{etaForCARMA}. In this case, Theorem~\ref{carmaCointTheorem} shows that $(X_t)_{t\in \mathbb{R}}$ given by \eqref{cointegratedCARMA} defines a solution to the corresponding MSDDE. As noted right after the formal CARMA equation \eqref{carmaEq}, the initial value $\xi$ should not affect whether $(X_t)_{t\in \mathbb{R}}$ can be thought of as a solution (since $P_p\xi =0$). Hence, suppose that $\xi = 0$. By heuristically computing $\mathcal{F}[X]$ from \eqref{cointegratedCARMA} we obtain
\begin{align*}
\mathcal{F}[X](y) &= (iy)^{-1}C_0 \mathcal{F}[DZ](y) + \mathcal{F}[C](y)\mathcal{F}[DZ](y)= P(iy)^{-1}Q(iy)\mathcal{F}[DZ](y)
\end{align*} 
for $y\in \mathbb{R}$ which, by multiplication of $P(iy)$, shows that $(X_t)_{t\in \mathbb{R}}$ solves \eqref{carmaEq}.
\end{remark}

\section{Proofs}\label{proofs}

\begin{proof}[Proof of Proposition~\ref{ecf_result_intro}]
	We start by arguing that 
	\begin{equation}\label{biStep}
	\sup_{t\geq 0} e^{\varepsilon t}\lVert \pi (t) \rVert<\infty
	\end{equation}
	for a given $\varepsilon \in (0, \delta)$. Note that, for any given finite signed matrix-valued measure $\mu$ on $[0,\infty)$,
	\begin{equation}\label{integratedLaplace}
	\mathcal{L}[\mu](z) = \int_{[0,\infty)} e^{-zu}\, \mu (du) = z \int_0^\infty e^{-zu} \mu ([0,u])\, du
	\end{equation}
	for all $z\in \mathbb{C}$ with $\text{Re}(z)>0$ using integration by parts. Consequently, 
	\begin{equation}\label{trueLaplace}
	\mathcal{L}[\pi](z) = z^{-1}\mathcal{L}[\tilde{\eta}](z),\quad \text{Re}(z)>0,
	\end{equation}
	 using the notation $\tilde{\eta}=\eta-\Pi_0\delta_0$. On the other hand, $z\mapsto \mathcal{L}[\tilde{\eta}](z)$ is analytic on $\mathbb{H}_\delta$ (by \eqref{exponentialDecay_intro}), and since $\mathcal{L}[\tilde{\eta}](0) = \mathcal{L}[\eta](0) - \eta ([0,\infty)) = 0$, $z\mapsto z^{-1}\mathcal{L}[\tilde{\eta}](z)$ is also analytic on $\mathbb{H}_\delta$, and we deduce that
	\begin{equation*}
	C:= \sup_{\text{Re}(z)\geq -\tilde{\varepsilon}} \lVert \mathcal{L}[\tilde{\eta}](z) \rVert + \sup_{\text{Re}(z)\geq -\tilde{\varepsilon}} \bigl\lVert z^{-1} \mathcal{L}[\tilde{\eta}](z) \bigr\rVert<\infty
	\end{equation*}
	for an arbitrary $\tilde{\varepsilon} \in (\varepsilon,\delta)$. Hence, we find that
	\begin{equation*}
	\sup_{\text{Re}(z)>-\tilde{\varepsilon}}\int_\mathbb{R} \bigl\lVert z^{-1}\mathcal{L}[\tilde{\eta}](z) \bigr\rVert^2\, d\text{Im} (z) \leq \Bigr(2+\int_{[-1,1]^c}y^{-2}\, dy \Bigl) C^2<\infty,
	\end{equation*}
	and it follows by \cite[Lemma~4.1]{contARMA} (or a slight modification of \cite[Theorem~1 (Section~3.4)]{dymMK}) and \eqref{trueLaplace} that $t \mapsto e^{\tilde{\varepsilon} t}\lVert \pi (t) \rVert$ belongs to $L^2$. For \eqref{biStep} to be satisfied it suffices to argue that $\sup_{t\geq 0}e^{\varepsilon t}\vert \pi_{ij}(t)\vert < \infty$, where $\pi_{ij}$ refers to an arbitrarily chosen entry of $\pi$. Using integration by parts we find that
	\begin{equation}\label{ibpPi}
	e^{\varepsilon t}\vert \pi_{ij}(t)\vert \leq \vert \pi_{ij}(0)\vert + \int_0^\infty e^{\varepsilon u}\, \vert \tilde{ \eta}_{ij}\vert (du) + \varepsilon \int_0^\infty e^{\varepsilon u}\vert \pi_{ij}(u)\vert \, du.
	\end{equation}
	It is clear that the first term on the right-hand side of \eqref{ibpPi} is finite, and the same holds for the second term by \eqref{exponentialDecay_intro}. For the last term we use the Cauchy-Schwarz inequality and the fact that $(u\mapsto e^{\tilde{\varepsilon} u}\pi_{ij}(u))\in L^2$ to deduce
	\begin{equation*}
	\Bigl(\int_0^\infty e^{\varepsilon u} \vert \pi_{ij} (u)\vert\, du\Bigr)^2 \leq \int_0^\infty e^{-2(\tilde{\varepsilon}-\varepsilon)u}\, du\, \int_0^\infty \bigl(e^{\tilde{\varepsilon}u}\pi_{ij}(u) \bigr)^2\, du<\infty
	\end{equation*}
	and this ultimately allows us to conclude that \eqref{biStep} holds. To show \eqref{errorCorrectionForm_main} it suffices to argue that
	\begin{equation}\label{mainStep}
	\biggl(\int_s^t X^T\ast \tilde{\eta}^T (u)\, du\biggr)^T = \int_0^\infty \pi (u) (X_{t-u}-X_{s-u})\, du
	\end{equation}
	almost surely for each $s<t$. Using that $\tilde{\eta}$ coincides with the Lebesgue-Stieltjes measure of $\pi$, together with integration by parts on the functions $v\mapsto \pi (v)$ and $v\mapsto \int_{s-v}^{t-v}X_u\, du$, we obtain
	\begin{align}\label{convergence}
	\begin{aligned}
	\biggl(\int_s^t X^T\ast \tilde{\eta}^T (u)\, du\biggr)^T &= \lim_{N\to \infty}\biggl(\int_{[0,N]}\biggl(\int_{s-v}^{t-v}X_u\, du\biggr)^T\, \pi^T (dv)\biggr)^T\\
	&= \lim_{N\to \infty} \biggl(\pi (N) \int_{s-N}^{t-N} X_u\, du + \int_0^N \pi (u)(X_{t-u}-X_{s-u})\, du\biggr).
	\end{aligned}
	\end{align}
	By \cite[Corollary~A.3]{QOU}, since $(X_t)_{t\in \mathbb{R}}$ has stationary increments and $\mathbb{E}\lVert X_t \rVert<\infty$, there exist $\alpha,\beta >0$ such that $\mathbb{E}\lVert X_u \rVert\leq \alpha + \beta \vert u \vert$ for all $u\in \mathbb{R}$. Consequently, we may as well find $\alpha^*,\beta^*>0$ (depending on $s$ and $t$) which satisfy 
	\begin{equation*}
	\mathbb{E}\, \biggl\lVert\int_{s-N}^{t-N} X_{u}\, du \biggr\rVert\leq \alpha^* + \beta^* N.
	\end{equation*}
	From this inequality, and due to \eqref{biStep}, each entry of $\pi (N)\int_{s-N}^{t-N} X_u\, du$ converges to $0$ in $L^1(\mathbb{P})$ as $N \to \infty$. The same type of reasoning gives that
	\begin{equation*}
	\mathbb{E}\int_0^\infty \lVert \pi (u) (X_{t-u}-X_{s-u}) \rVert\, du <\infty,
	\end{equation*}
	showing that each entry of $u\mapsto \pi (u)(X_{t-u}-X_{s-u})$ is almost surely integrable with respect to the Lebesgue measure and, hence, \eqref{convergence} implies \eqref{mainStep}. Finally, we need to argue that if $(X_t)_{t\in \mathbb{R}}$ is a solution to \eqref{MSDDE}, $(\Pi_0 X_t)_{t\in \mathbb{R}}$ is stationary. Since $(X_t,Z_t)_{t\in \mathbb{R}}$ has stationary increments it follows immediately from \eqref{errorCorrectionForm_main} that $V_t^\lambda := \lambda^{-1}\int_t^{t+\lambda}\Pi_0 X_u\, du$, $t\in \mathbb{R}$, is a stationary process for any $\lambda >0$. Since $(X_t)_{t\in \mathbb{R}}$ has stationary increments and $\mathbb{E}\lVert X_t \rVert < \infty$, it is continuous in $L^1(\mathbb{P})$ (see \cite[Corollary~A.3]{QOU}), and hence $V_t^\lambda$ converges to $\Pi_0 X_t$ in $L^1(\mathbb{P})$ as $\lambda \downarrow 0$ for any $t\in \mathbb{R}$. Consequently, $(\Pi_0 X_t)_{t\in \mathbb{R}}$ is stationary as well, and this finishes the proof.
	
\end{proof}

\begin{proof}[Proof of Proposition~\ref{propEqOfAssump}]
	Assume that we are in case \eqref{firstProp}. If $z\mapsto h_\eta (z)^{-1}$ has no poles at all, then $\det h_\eta (z) = 0$ implies $\text{Re}(z)<0$ and the rank of $\Pi_0$ is $n$, and thus case \eqref{secondProp} is satisfied as well. If $z\mapsto h_\eta (z)^{-1}$ has a simple pole at $0$, the rank $r$ of $\Pi_0 = -h_\eta (0)$ is strictly less than $n$, and the residue formula in \cite{schumacher1991system} implies that $(\alpha^\perp)^T M\beta^\perp$ is invertible, where
	\begin{equation*}
	M := \frac{h_\eta (z)+\eta ([0,\infty))}{z}\Bigl\vert_{z=0} = I_n - \Pi ([0,\infty))
	\end{equation*}
	is the derivative of $h_\eta$ at $0$, and $\alpha^\perp,\beta^\perp\in \mathbb{R}^{n\times (n-r)}$ are any two matrices of rank $n-r$ satisfying $\Pi_0^T\alpha^\perp = \Pi_0 \beta^\perp=0$. Conversely, if we are in case \eqref{secondProp}, the facts that the zeroes of $z\mapsto \det h_\eta (z)$ are isolated points in $\{z\in \mathbb{C}\, :\, \text{Re}(z)>-\varepsilon\}$ and $\vert \det h_\eta (z)\vert \neq 0$ for $\vert z \vert$ sufficiently large ensure the existence of a $\delta \in (0,\varepsilon]$ such that $\det h_\eta (z) \neq 0$ for all $z\in \mathbb{H}_\delta \setminus \{0\}$. If the rank $r$ of $\Pi_0$ is $n$, $z\mapsto h_\eta (z)^{-1}$ has no poles at all on $\mathbb{H}_\delta$, and if $r<n$ and $(\alpha^\perp)^T M \beta^\perp$ is invertible, the residue formula in \cite{schumacher1991system} implies that $z\mapsto h_\eta (z)^{-1}$ has a simple pole at $0$.
\end{proof}

We will now turn to the construction of a solution to \eqref{MSDDE}. Lemma~\ref{fExistence} concerns the existence of the function $f$ introduced in Theorem~\ref{existence} and its properties.
\begin{lemma}\label{fExistence}
	Suppose that Condition~\ref{mostIone} holds. Then there exists a unique function $f:\mathbb{R}\to \mathbb{R}^{n\times n}$ enjoying the following properties:
	\begin{enumerate}[(i)]
		\item\label{prop2} $\sup_{t\geq 0} e^{\varepsilon t}\lVert f(t) \rVert  <\infty$ for all $\varepsilon < \delta$.
		\item\label{prop3} $\mathcal{L}[f](z) =I_n- zh_\eta (z)^{-1}$ for all $z\in \mathbb{H}_\delta$.
		\item\label{prop1} $f (t) = 0$ for $t<0$ and $f(t) = \int_0^t f \ast \eta (u)\, du-\eta ([0,t])$ for $t \geq 0$.
	\end{enumerate}
\end{lemma}

\begin{proof}
	First note that, by assumption, $z\mapsto  I_n - zh_\eta (z)^{-1}$ is an analytic function on $\mathbb{H}_\delta$. For any $\varepsilon \in (0, \delta)$ we will argue that
	\begin{equation}\label{keyKernelArgument}
	\sup_{\text{Re}(z)>-\varepsilon}\int_\mathbb{R} \bigl\lVert I_n - z h_\eta (z)^{-1} \bigr\rVert^2\, d\text{Im}(z)<\infty.
	\end{equation}
	If this is the case, a slight extension of the characterization of Hardy spaces (see \cite[Lemma~4.1]{contARMA} or \cite[Theorem~1 (Section~3.4)]{dymMK}) ensures the existence of a function $f:\mathbb{R}\to \mathbb{C}^{n\times n}$, vanishing on $(-\infty,0)$, such that each entry of $t\mapsto e^{\varepsilon t}f(t)$ belongs to $L^2$ and $\mathcal{L}[f](z) = I_n - zh_\eta (z)^{-1}$ for all $z\in \mathbb{C}$ with $\text{Re}(z)>-\varepsilon$. Since $\varepsilon$ was arbitrary and, by uniqueness of the Laplace transform, the relation holds true for all $z \in \mathbb{H}_\delta$. Moreover, since $\overline{\mathcal{F}[f](-y)} = \mathcal{F}[f](y)$ for all $y\in \mathbb{R}$ ($\overline{z}$ denoting the complex conjugate of $z\in \mathbb{C}$), $f$ takes values in $\mathbb{R}^{n\times n}$. To show \eqref{keyKernelArgument} observe initially that
	\begin{equation*}
	C_1 := \sup_{\text{Re}(z)\geq -\varepsilon} \lVert \mathcal{L}[\eta](z) \rVert<\infty,
	\end{equation*}
	since $e^{\varepsilon t}\, \vert\eta_{ij}\vert (dt)$ is a finite measure for all $i,j=1,\dots, n$. The same fact ensures that
	\begin{enumerate}[(i)]
		\item the absolute value of the determinant of $h_\eta (z)$ behaves as $\vert z \vert^n$ as $\vert z \vert \to \infty$, and
		\item the dominating cofactors of $h_\eta (z)$ as $\vert z \vert \to \infty$ are those on the diagonal (the $(i,i)$-th cofactor, $i=1,\dots, n$) and their absolute values behave as $\vert z\vert^{n-1}$ as $\vert z \vert \to \infty$.
	\end{enumerate} 
	In particular, $\lVert h_\eta (z)^{-1} \rVert$ behaves as $\vert z \vert^{-1}$ as $\vert z \vert \to \infty$ and, hence,
	\begin{equation}\label{boundedInverse}
	C_2 := \sup_{\text{Re}(z)\geq -\varepsilon}\lVert zh_\eta (z)^{-1} \rVert<\infty.
	\end{equation}
	Consequently, for any $z\in \mathbb{C}$ with $\text{Re}(z)\geq - \varepsilon$,
	\begin{equation*}
	\int_{[-1,1]} \lVert I_n - z h_\eta (z)^{-1} \rVert^2\, d\text{Im}(z) \leq 2 \bigl(\sqrt{n}+ C_2\bigr)^2
	\end{equation*}
	and
	\begin{align*}
	\int_{[-1,1]^c}\lVert I_n - z h_\eta (z)^{-1} \rVert^2\, d\text{Im}(z)&\leq C_1^2  \int_{[-1,1]^c}\lVert h_\eta (z)^{-1} \rVert^2\, d\text{Im}(z) \\
	&\leq (C_1C_2)^2 \int_{[-1,1]^c} \vert x\vert^{-2}\, dx
	\end{align*}
	using $I_n - zh_\eta (z)^{-1} = -h_\eta (z)^{-1}\mathcal{L}[\eta](z)$ and that $\lVert \cdot \rVert$ is a submultiplicative norm. This verifies \eqref{keyKernelArgument} and, hence, proves the existence of a function $f:\mathbb{R}\to \mathbb{R}^{n\times n}$ with $f(t) = 0$ for $t<0$ and $\mathcal{L}[f](z) = I_n-zh_\eta (z)^{-1}$ for $z\in \mathbb{H}_\delta$ (in particular, verifying \eqref{prop3}). To show \eqref{prop1}, note that
	\begin{equation}\label{functionalRelation}
	\mathcal{L}[f\ast \eta] (z)-\mathcal{L}[\eta](z) = -zh_\eta (z)^{-1}\mathcal{L}[\eta](z) = z\mathcal{L}[f](z)
	\end{equation}
	for $z\in \mathbb{H}_\delta$. By using the observation in \eqref{integratedLaplace} on the measures $f\ast \eta (u)\, du$ and $\eta$ together with \eqref{functionalRelation} we establish that
	\begin{equation}\label{aeverywhere}
	f(t) = \int_0^t f \ast \eta (u)\, du - \eta ([0,t])
	\end{equation}
	for almost all $t\geq 0$. Since we can choose $f$ to satisfy \eqref{aeverywhere} for all $t\geq 0$ without modifying its Laplace transform, we have established \eqref{prop1}. By the càdlàg property of $f$, the uniqueness part follows as well. Finally, we need to argue that \eqref{prop2} holds, and for this it suffices to argue that $\sup_{t\geq 0} e^{\varepsilon t}\vert f_{ij}(t)\vert<\infty$ for all $\varepsilon \in (0,\delta)$ where $f_{ij}$ refers to an arbitrarily chosen entry of $f$. From \eqref{aeverywhere} it follows that the Lebesgue-Stieltjes measure of $f_{ij}$ is given by $\sum_{k=1}^n f_{ik}\ast \eta_{kj} (u)\, du -\eta_{ij} (du)$. Therefore, integration by parts yields
	\begin{align}\label{ibp}
	\begin{aligned}
	e^{\varepsilon t} \vert f_{ij}(t)\vert \leq &\vert f_{ij} (0)\vert + \sum_{k=1}^n \int_0^\infty e^{\varepsilon u}\vert f_{ik}\vert \ast \vert \eta_{kj}\vert (u)\, du +\int_0^\infty e^{\varepsilon u}\, \vert \eta_{ij}\vert (du)  \\
	& + \varepsilon \int_0^\infty e^{\varepsilon u}\vert f_{ij}(u)\vert\, du,
	\end{aligned}
	\end{align}
	so to prove the result we only need to argue that each term on right-hand side of \eqref{ibp} is finite. The assumption \eqref{exponentialDecay_intro} implies immediately that $\int_0^\infty e^{\varepsilon u}\, \vert \eta_{ij}\vert (du)<\infty$. As noted in the beginning of the proof, $u\mapsto e^{\varepsilon' u}f_{ij}(u)$ belongs to $L^2$ for an arbitrary $\varepsilon'\in (0,\delta)$. In particular, for $\varepsilon'\in (\varepsilon,\delta)$,
	\begin{equation*}
	\int_0^\infty e^{\varepsilon u}\vert f_{ij}(u)\vert \, du \leq \Bigl(\int_0^\infty e^{-2(\varepsilon' - \varepsilon)u}\, du\, \int_0^\infty \bigl(e^{\varepsilon' u}f_{ij}(u) \bigr)^2\, du \Bigr)^{1/2}<\infty.
	\end{equation*}
	Finally, since $\int_0^\infty e^{\varepsilon u}\vert f_{ik}\vert \ast \vert \eta_{kj}\vert (u)\, du = \int_{[0,\infty)}e^{\varepsilon u}\, \vert \eta_{kj}\vert (du)\, \int_0^\infty e^{\varepsilon u}\vert f_{ik}(u)\vert\, du$, it follows by the former arguments that this term is finite as well, and this concludes the proof.
\end{proof}
\begin{remark}\label{relationToOrdMSDDE}
	Suppose that $\det h_\eta(z) \neq 0$ for all $z\in \mathbb{H}_\delta$ so that Condition~\ref{mostIone} is satisfied and $z\mapsto h_\eta (z)^{-1}$ has no poles. Under this assumption it was argued in \cite[Proposition~5.1]{basse2018multivariate} that there exists a function $g:\mathbb{R}\to \mathbb{R}^{n\times n}$, which is vanishing on $(-\infty,0)$, is absolutely continuous on $[0,\infty)$ and decays exponentially fast at $\infty$, such that $\mathcal{L}[g](z) = h_\eta (z)^{-1}$ for $z\in \mathbb{H}_\delta$. Since property \eqref{prop3} implies $\mathcal{L}[f](z) = -h_\eta (z)^{-1}\mathcal{L}[\eta](z)$, it must be the case that $f=-g\ast \eta$.
\end{remark}

\begin{proof}[Proof of Theorem~\ref{existence}]
	The existence of $f$ is covered by Lemma~\ref{fExistence}. According to \cite[Corollary~A.3]{QOU} and by equivalence of matrix norms, we may choose $\alpha, \beta,\gamma >0$ such that $\mathbb{E}\lVert Z_t \rVert\leq \alpha + \beta \vert t\vert$ for all $t\in \mathbb{R}$ and $\sum_{i,j=1}^n \vert a_{ij}\vert \leq \gamma \lVert A \rVert$ for all $A = (a_{ij})\in \mathbb{R}^{n\times n}$. Using this together with property Lemma~\ref{fExistence}\eqref{prop2}, we obtain that
	\begin{align*}
	\mathbb{E}\int_0^\infty \lVert f(u)Z_{t-u}\rVert\, du  
	\leq (\alpha + \beta \vert t \vert)\gamma \int_0^\infty \lVert f(u)\rVert\, du + \beta\gamma \int_0^\infty \lVert f(u)\rVert \vert u \vert\, du<\infty.
	\end{align*}
	In particular, this shows that $u\mapsto f(u) Z_{t-u}$ belongs to $L^1$ almost surely and, hence, $(X_t)_{t\in \mathbb{R}}$ given by \eqref{solution} is a well-defined process. We will now split the proof in two parts: first, we argue that $(X_t)_{t\in \mathbb{R}}$ given by \eqref{solution} is indeed a solution to \eqref{MSDDE} (existence) and, next, we show that any other solution necessarily admits this representation (uniqueness).
	
	\textbf{Existence:} Note that $\mathbb{E}\lVert Z_t\rVert^2\leq \gamma_1 + \gamma_2t^2$ for all $t$ and suitable $\gamma_1,\gamma_2>0$ by \cite[Corollary~A.3]{QOU}, so we may use similar reasoning as above to deduce that $\mathbb{E}\lVert X_t\rVert^2 < \infty$ for all $t\in\mathbb{R}$. Moreover, since $(X_t)_{t\in \mathbb{R}}$ solves \eqref{MSDDE} if and only if it solves \eqref{errorCorrectionForm_main}, we may and do assume $\xi = 0$ so that
	\begin{equation*}
	X_t = Z_t- \int_0^\infty f(u)Z_{t-u}\, du,\quad t \in \mathbb{R}.
	\end{equation*}
	To show that $(X_t)_{t\in \mathbb{R}}$ satisfies \eqref{MSDDE}, we need to argue that
	\begin{equation}\label{solutionRelation}
	X_t - X_s - (Z_t - Z_s) = \int_s^t\eta \ast X (u)\, du
	\end{equation}
	for $s<t$.  To this end, note that
	\begin{equation}\label{arg1}
	X_t - X_s - (Z_t - Z_s) = \int_\mathbb{R} \eta ((s-u,t-u])Z_u\, du - \int_\mathbb{R} \int_{s-u}^{t-u} f \ast \eta (v)\, dv\, Z_u\, du
	\end{equation}
	and
	\begin{align}
	\begin{aligned}\label{arg2}
	\int_s^t \eta \ast X (u)\, du &= \int_\mathbb{R}\eta ((s-u,t-u])X_u\, du \\
	&= \int_\mathbb{R}\eta ((s-u,t-u]) Z_u\, du -\int_\mathbb{R}\eta ((s-u,t-u])\int_\mathbb{R} f(v)Z_{u-v}\, dv\, du
	\end{aligned}
	\end{align}
	using Lemma~\ref{fExistence}\eqref{prop1} and \eqref{solution}, respectively. Moreover, by comparing their Laplace transforms, one can verify that $\eta \ast f := (f^T\ast \eta^T)^T = f\ast \eta$ and, thus,
	\begin{align}\label{arg3}
	\begin{aligned}
	\int_\mathbb{R} \int_{s-u}^{t-u} f \ast \eta (v)\, dv\, Z_u\, du &= \int_\mathbb{R}\int_\mathbb{R}\eta ((s-u-v,t-u-v]) f(v)\, dv\, Z_u \, du\\
	&= \int_\mathbb{R}\eta((s-u,t-u])\int_\mathbb{R} f(v)Z_{u-v}\, dv\, du
	\end{aligned}
	\end{align}
	It follows by combining \eqref{arg1}-\eqref{arg3} that \eqref{solutionRelation} is satisfied. Recall that, for $(X_t)_{t\in \mathbb{R}}$ to be a solution, we need to argue that $(X_t,Z_t)_{t\in \mathbb{R}}$ has stationary increments. However, since
	\begin{equation*}
	X_{t+h}-X_h = (Z_{t+h}-Z_h) - \int_0^\infty f(u)[(Z_{t-u+h}-Z_h)- (Z_{-u+h}-Z_h)]\, du ,\quad t \in \mathbb{R},
	\end{equation*}
	and the distribution of $(Z_{t+h}-Z_h)_{t\in \mathbb{R}}$ does not depend on $h$, it follows that the distribution of $(X_{t+h}-X_h,Z_{t+h}-Z_h)_{t\in \mathbb{R}}$ does not depend on $h$. A rigorous argument can be carried out by approximating the above Lebesgue integral by Riemann sums in $L^1(\mathbb{P})$; since this procedure is similar to the one used in the proof of \cite[Theorem~3.1]{basse2018multivariate}, we omit the details here.
	
	\textbf{Uniqueness:} Suppose that $(Y_t)_{t\in \mathbb{R}}$ satisfies \eqref{MSDDE}, $\mathbb{E}\lVert Y_t \rVert^2<\infty$ for all $t\in \mathbb{R}$, and $(Y_t,Z_t)_{t\in \mathbb{R}}$ has stationary increments. In addition, suppose for the moment that we have already shown that
	\begin{equation}\label{sameIncrements}
	Y_t-Y_s = X_t-X_s,\quad s,t\in \mathbb{R}.
	\end{equation}
	Then it follows from \eqref{errorCorrectionForm_main} that $V_\lambda := \lambda^{-1}\int_0^\lambda \Pi_0 (Y_u-X_u)\, du = 0$ almost surely for all $\lambda >0$. On the other hand, since $(X_t)_{t\in \mathbb{R}}$ and $(Y_t)_{t\in \mathbb{R}}$ have stationary increments, they are continuous in $L^1(\mathbb{P})$ and, hence, $V_\lambda \to \Pi_0 (Y_0-X_0)$ in $L^1(\mathbb{P})$ as $\lambda \downarrow 0$. This shows that $Y_0 - X_0$ belongs to the null space of $\Pi_0$ almost surely and, consequently, $(Y_t)_{t\in \mathbb{R}}$ is necessarily of the form \eqref{solution}. The remaining part of the proof concerns showing \eqref{sameIncrements} or, equivalently, the process $\Delta_h Y_t :=Y_t-Y_{t-h}$, $t\in \mathbb{R}$, is unique for any $h>0$. We will rely on the same type of ideas as in the proof of \cite[Proposition~7]{comte1999discrete} and \cite[Proposition~4.5]{fasen2016cointegrated}. Suppose first that $\Pi_0$ has reduced rank $r\in (0,n)$ and let $\alpha, \beta \in \mathbb{R}^{n\times r}$ be a rank decomposition of $\Pi_0$ as in Remark~\ref{stationarityHeuristics}. Moreover, let $\alpha^\perp,\beta^\perp\in \mathbb{R}^{n\times (n-r)}$ be matrices of rank $n-r$ such that $\alpha^T \alpha^\perp = \beta^T\beta^\perp=0$. Then it follows from Theorem~\ref{ecf_result_intro} that
	\begin{align}
	\label{uniquenessProof}
	\begin{aligned}
	\alpha^T \Delta_h Y_t &= \alpha^T \alpha \beta^T\int_0^h Y_{t-u}\, du+ \alpha^T \int_0^\infty \pi (u) \Delta_h Y_{t-u}\, du + \alpha^T \Delta_h Z_t,\\
	(\alpha^\perp)^T\Delta_h Y_t &= (\alpha^\perp)^T \int_0^\infty \pi (u) \Delta_h Y_{t-u}\, du + (\alpha^\perp)^T\Delta_h Z_t
	\end{aligned}
	\end{align}
	for each $t\in \mathbb{R}$.	Define the stationary processes $U_t = (\beta^T \beta)^{-1} \beta^T Y_t$ and $V_t = \bigl((\beta^\perp)^T \beta^\perp\bigr)^{-1} (\beta^\perp)^T \Delta_h Y_t$ (cf. Theorem~\ref{ecf_result_intro}). By using that $\Delta_h Y_t = \beta \Delta_h U_t + \beta^\perp V_t$ and rearranging terms, \eqref{uniquenessProof} can be written as
	\begin{equation}\label{comprimisedSystem}
	\mu \ast \begin{bmatrix} U_t \\ V_t
	\end{bmatrix}=\tilde{Z}_t,\quad t \in \mathbb{R},
	\end{equation}
	where
	\begin{equation*}
	\mu = \begin{bmatrix}
	\alpha^T \bigl[(\delta_0 - \delta_h)I_n - \alpha \beta^T\mathds{1}_{(0,h]}(u)\, du - \Delta_h\pi (u)\, du\bigr]\beta & \alpha^T[\delta_0 I_n - \pi (u)\, du]\beta^\perp \\
	(\alpha^\perp)^T [(\delta_0 - \delta_h)I_n - \Delta_h \pi (u)\, du]\beta & (\alpha^\perp)^T[\delta_0 I_n - \pi (u)\, du]\beta^\perp
	\end{bmatrix}
	\end{equation*}
	and $\tilde{Z}_t = [\Delta_h Z_t^T\alpha,\,  \Delta_h Z_t^T \alpha^\perp]^T$. Now, note that the Fourier transform $\mathcal{F}[\mu]$ of $\mu$ takes the form
	\begin{align*}
	\MoveEqLeft\mathcal{F}[\mu](y) \\
	&=\begin{bmatrix} \alpha^T\bigl[(1-e^{-ihy})[I_n - \mathcal{F}[\pi](y)] - \alpha\beta^T \mathcal{F}[\mathds{1}_{(0,h]}](y)\bigr]\beta & \alpha^T [I_n - \mathcal{F}[\pi](y)]\beta^\perp \\
	(\alpha^\perp)^T (1-e^{-ihy})[I_n -\mathcal{F}[\pi](y)]\beta & (\alpha^\perp)^T [I_n-\mathcal{F}[\pi](y)]\beta^\perp \end{bmatrix}.
	\end{align*}
	In particular, it follows that
	\begin{align*}
	\det \mathcal{F}[\mu](0) &= \det \begin{bmatrix} -\alpha^T \alpha \beta^T\beta h & \alpha^T [I_n-\mathcal{F}[\pi](0)]\beta^\perp \\
	0 & (\alpha^\perp)^T [I_n-\mathcal{F}[\pi](0)]\beta^\perp \end{bmatrix}\\
	&= (-h)^r \det (\alpha^T\alpha) \det (\beta^T\beta) \det \bigl((\alpha^\perp)^T[I_n-\Pi ([0,\infty))]\beta^\perp\bigr),
	\end{align*}
	which is non-zero by Proposition~\ref{propEqOfAssump}. Consequently, it follows from \eqref{comprimisedSystem} that the means of $(U_t)_{t\in \mathbb{R}}$ and $(V_t)_{t\in \mathbb{R}}$ are uniquely determined by the one of $(\tilde{Z}_t)_{t\in \mathbb{R}}$, $[\mathbb{E}U_0^T,\mathbb{E}V_0^T]^T = \mu ([0,\infty))^{-1}\mathbb{E}\tilde{Z}_0$. For this reason we may without loss of generality assume that $(U_t)_{t\in \mathbb{R}}$, $(V_t)_{t\in \mathbb{R}}$ and $(\tilde{Z}_t)_{t\in \mathbb{R}}$ are all zero mean processes so that they admit spectral representations. Recall that the spectral representation of a stationary, square integrable and zero mean process $(S_t)_{t\in \mathbb{R}}$ is given by $S_t = \int_\mathbb{R}e^{ity}\, \Lambda_S (dy)$, $t\in \mathbb{R}$, where $(\Lambda_S (t))_{t\in \mathbb{R}}$ is a complex-valued spectral process which is square integrable and continuous in $L^2(\mathbb{P})$, and which has orthogonal increments. (Integration with respect to $\Lambda_S$ can be defined as in \cite[pp 388-390]{grimmett2001probability} for all functions in $L^2(F_S)$, $F_S$ being the spectral distribution of $(S_t)_{t\in \mathbb{R}}$.) Consequently, by letting $\Lambda_U$, $\Lambda_V$ and $\Lambda_{\tilde{Z}}$ be the spectral processes corresponding to $(U_t)_{t\in \mathbb{R}}$, $(V_t)_{t\in \mathbb{R}}$ and $(\tilde{Z}_t)_{t\in \mathbb{R}}$, equation \eqref{comprimisedSystem} can be rephrased as
	\begin{equation}\label{integralRelation}
	\int_\mathbb{R}e^{ity} \mathcal{F}[\mu](y)\, \begin{bmatrix} \Lambda_U \\ \Lambda_V \end{bmatrix}(dy) = \int_\mathbb{R}e^{ity}\, \Lambda_{\tilde{Z}}(dy),\quad t \in \mathbb{R}.
	\end{equation}
	Here we have used a stochastic Fubini result for spectral processes, e.g., \cite[Proposition~A.1]{davis2018stochastic}. Since the functions $y \mapsto e^{ity}$, $t\in \mathbb{R}$, are dense in $L^2(F)$ for any finite measure $F$ (cf. \cite[p. 150]{Yaglom:cor}), the relation \eqref{integralRelation} remains true when $y\mapsto e^{ity}$ is replaced by any measurable and, say, bounded function $g:\mathbb{R}\to \mathbb{C}^{n\times n}$. In particular, we will choose
	\begin{equation*}
	g(y) = e^{ity}(iy)h_\eta (iy)^{-1}\begin{bmatrix} \alpha^T \\ (\alpha^\perp)^T \end{bmatrix}^{-1}
	\end{equation*}
	for $y \neq 0$ and $g(0) = [0_{n\times r},\, \beta^\perp]\mathcal{F}[\mu](0)^{-1}$. Note that by \eqref{boundedInverse}, $g$ is indeed bounded. After observing that
	\begin{align*}
	\mathcal{F}[\mu](y) &= \begin{bmatrix} \alpha^T \\ (\alpha^\perp)^T \end{bmatrix} \begin{bmatrix} (1-e^{-ihy})\bigl[ I_n-\mathcal{F}[\pi](y)+ (iy)^{-1}\alpha\beta^T\bigr] & I_n - \mathcal{F}[\pi](y) \\
	(1-e^{-ihy})[I_n -\mathcal{F}[\pi](y)] & I_n-\mathcal{F}[\pi](y) \end{bmatrix} \begin{bmatrix} \beta & \beta^\perp \end{bmatrix}\\
	&= (iy)^{-1}\begin{bmatrix} \alpha^T \\ (\alpha^\perp)^T \end{bmatrix}  h_\eta (iy)  \begin{bmatrix} (1-e^{-ihy})\beta & \beta^\perp \end{bmatrix}
	\end{align*}
	for $y\neq 0$, it is easy to verify that $g(y)\mathcal{F}[\mu](y) = [\beta (e^{ity}-e^{i(t-h)y}),\, \beta^\perp e^{ity}]$ for all $y\in \mathbb{R}$. Consequently, it follows from \eqref{integralRelation} that
	\begin{equation*}
	\Delta_h Y_t = \int_\mathbb{R}\begin{bmatrix}\beta (e^{ity}-e^{i(t-h)y}) & \beta^\perp e^{ity} \end{bmatrix}\, \begin{bmatrix} \Lambda_U \\ \Lambda_V \end{bmatrix} (dy) = \int_\mathbb{R} g(y)\, \Lambda_{\tilde{Z}} (dy),
	\end{equation*}
	showing that the process $(\Delta_h Y_t)_{t\in \mathbb{R}}$ is uniquely determined by $(\tilde{Z}_t)_{t\in \mathbb{R}}$. Now we only need to ague that this type of uniqueness also holds when $\Pi_0$ is invertible and $\Pi_0 = 0$. If $\Pi_0$ is invertible, $(Y_t)_{t\in \mathbb{R}}$ must in fact be stationary (cf. Remark~\ref{stationarityHeuristics}), and by \cite[Theorem~3.1]{basse2018multivariate} there is only one process enjoying this property. If $\Pi_0 = 0$, the case is simpler than if $r\in (0,n)$, since here we only need to consider the second equation of \eqref{uniquenessProof} with $\alpha^\perp = I_n$ and the spectral representation of $(\Delta_h Y_t)_{t\in \mathbb{R}}$. To avoid too many repetitions we leave out the details.
\end{proof}

\begin{proof}[Proof of Corollary~\ref{regGranger}]
	As noted right before the statement we only need to argue that \eqref{GrangerLack} is satisfied with respect to the definition \eqref{integralDefine}.
	In order to do so, note that
	\begin{align*}
	\Bigr(\int_0^\infty f(u)[Z_t-Z_{t-u}]\, du \Bigl)_i
	&= \sum_{j=1}^n \int_0^\infty \mathcal{I}_j (\mathds{1}_{(t-u,t]}) f_{ij}(u)\, du  \\
	&= \sum_{j=1}^n \mathcal{I}_j \Bigl(\mathds{1}_{[0,\infty)}(t-\cdot)\int_{t-\cdot}^\infty f_{ij}(u)\, du\Bigr) \\ 
	&= \Bigl(\int_{-\infty}^t C(t-u)\, dZ_u\Bigr)_i,
	\end{align*}
	where $C(t) =0$ for $t<0$ and $C(t)=\int_t^\infty f(u)\, du$ for $t\geq 0$.
	Now observe that, for $z \in \mathbb{C}$ with $\text{Re}(z)<0$,
	 \begin{equation}\label{laplaceC}
	 \mathcal{L}[C](z) = z^{-1} \Bigl(\int_0^\infty f(t)\, dt - \mathcal{L}[f](z) \Bigr) = h_\eta (z)^{-1}-z^{-1}C_0
	 \end{equation}
	 using Remark~\ref{C0_structure} and Lemma~\ref{fExistence}\eqref{prop3}. Since both sides of \eqref{laplaceC} are analytic functions on $\mathbb{H}_\delta$, the equality holds true on $\mathbb{H}_\delta$. This proves that $C$ can be characterized as in the statement of Theorem~\ref{Granger_intro} and, thus, finishes the proof. 
	  
\end{proof}
\begin{remark}\label{Cdyn}
	As was the case for the function $f$ of Lemma~\ref{fExistence}, $C$ can also be obtained as a solution to a multivariate delay differential equation. Specifically, the shifted function $\tilde{C}(t) = C_0 + C(t)$, $t\geq 0$, satisfies
	\begin{equation}\label{Cdyn_concrete}
	\tilde{C}(t) - \tilde{C}(s) = \int_s^t \tilde{C}\ast \eta (u)\, du,\quad 0\leq s<t.
	\end{equation}
	By Theorem~\ref{existence} the initial condition is $\tilde{C}(0) = I_n$. To see that \eqref{Cdyn_concrete} holds note that, for fixed $0\leq s <t$, Lemma~\ref{fExistence}\eqref{prop1} implies
	\begin{equation*}
	\tilde{C}(t) - \tilde{C}(s) = - \int_s^t f(u)\, du = \int_s^t \Bigl(\eta ([0,u])-\int_0^u f \ast \eta (v)\, dv\Bigr)\, du,
	\end{equation*}
	and
	\begin{equation*}
	\int_0^u f\ast \eta (v)\, dv = \int_{[0,\infty)} \int_0^{u-r} f(v)\, dv \, \eta (dr) = \eta ([0,u])- \tilde{C} \ast \eta (u)
	\end{equation*}
	by Fubini's theorem. In the same way as in the proof of Theorem~\ref{ecf_result_intro}, one can rely on integration by parts to write \eqref{Cdyn_concrete} in error correction form:
	\begin{equation*}
	\tilde{C}(t) - \tilde{C}(s) = \int_s^t \tilde{C}(u)\Pi_0\, du + \int_0^\infty [\tilde{C}(t-u)-\tilde{C}(s-u)]\, \Pi (du),\quad 0\leq s <t.
	\end{equation*}
\end{remark}

\begin{proof}[Proof of Theorem~\ref{Granger_intro}]
	In view of Proposition~\ref{propEqOfAssump} we may assume that Condition~\ref{mostIone} is satisfied. Consequently, by using \cite[Example~4.2]{basse2018multivariate}, which states that an $n$-dimensional L\'{e}vy process with finite first moments is a regular integrator (that is, there exist $\mathcal{I}_1,\dots, \mathcal{I}_n$ satisfying Corollary~\ref{regGranger}\eqref{reg1}-\eqref{reg2}), the result is an immediate consequence of Corollary~\ref{regGranger}.
\end{proof}

\begin{proof}[Proof of Theorem~\ref{carmaCointTheorem}]
	Note that, by Condition~\ref{cointAssump}\eqref{invAssump}, we can choose $\varepsilon>0$ such that $\det Q(z) \neq0$ whenever $\text{Re}(z)\geq -\varepsilon$. To show that \eqref{etaForCARMA} is well-defined and satisfies \eqref{exponentialDecay_intro} for some $\delta>0$ it suffices to establish
	\begin{equation}\label{CARMAkeyCondition}
	\sup_{\text{Re}(z)>-\varepsilon}\int_\mathbb{R}\lVert I_n z-\eta_0-Q(z)^{-1}P(z)\rVert^2\, d\text{Im}(z) < \infty.
	\end{equation}
	(See, e.g., the beginning of the proof of Lemma~\ref{fExistence}.) It is straightforward to verify that $\eta_0 = Q_1-P_1$ is chosen such that $z\mapsto Q(z)(I_nz-\eta_0) - P(z)$ is a polynomial of at most order $p-2$. Consequently, the integrand in \eqref{CARMAkeyCondition} is of the form $\lVert Q(z)^{-1}R(z) \rVert^2$, where $Q$ is of strictly larger degree than $R$, and hence it follows by sub-multiplicativity of $\lVert \cdot \rVert$ that it decays at least as fast as $\vert z\vert^{-2}$ when $\vert z \vert \to \infty$. Since the integrand is also bounded on compact subsets of $\{z\in\mathbb{C}\, :\, \text{Re}(z)\leq \varepsilon \}$ we conclude that \eqref{CARMAkeyCondition} is satisfied. 
	
	Next, we will show that the assumptions of Theorem~\ref{Granger_intro} are satisfied (which, by Proposition~\ref{propEqOfAssump}, is equivalent to showing that Condition~\ref{mostIone} holds). Observe that $h_\eta (z) = Q(z)^{-1}P(z)$ when $\text{Re}(z)>-\varepsilon$, so by \eqref{firstCond} and \eqref{invAssump} in Condition~\ref{cointAssump} it follows that $\det h_\eta (z)= 0$ implies $\text{Re}(z)<0$ or $z=0$. Now, a Taylor expansion of $z \mapsto Q(z)^{-1}$ around $0$ yields 
	\begin{equation*}
	\mathcal{L}[\eta](z) = \eta ([0,\infty)) + \bigl(I_n+Q_{p-1}^{-1}Q_{p-2}Q_{p-1}^{-1}P_p-Q^{-1}_{p-1}P_{p-1} \bigr)z + O(z^2),\quad \vert z \vert \to 0,
	\end{equation*}
	and hence
	\begin{equation*}
	\Pi ([0,\infty)) = \frac{\mathcal{L}[\eta](z)-\eta ([0,\infty))}{z}\Bigr\vert_{z=0} = I_n-Q^{-1}_{p-1}\bigl[P_{p-1}-Q_{p-2}Q_{p-1}^{-1}P_p\bigr].
	\end{equation*}
	Let $\tilde{\alpha}=Q_{p-1}^T\alpha^\perp$ and $\tilde{\beta} = \beta^\perp$, and note that these matrices are of rank $n-r$ and satisfy $\Pi_0^T\tilde{\alpha} = \Pi_0 \tilde{\beta}=0$. Thanks to Condition~\ref{cointAssump}\eqref{midCond}, the matrix
	\begin{equation*}
	\tilde{\alpha}^T (I_n-\Pi ([0,\infty)))\tilde{\beta} = (\alpha^\perp)^TP_{p-1}\beta^\perp,
	\end{equation*}
	is invertible, so the assumptions of Theorem~\ref{Granger_intro} are satisfied.
	The remaining statements are now simply consequences of Corollary~\ref{regGranger}.
\end{proof}

\subsection*{Acknowledgments}
I would like to thank Andreas Basse-O'Connor and Jan Pedersen for helpful comments. This work was supported by the Danish Council for Independent Research (Grant DFF - 4002--00003).

\end{document}